\documentclass[oneside,english,reqno]{amsart}
\usepackage[T1]{fontenc}
\usepackage[latin9]{inputenc}
\usepackage{babel}
\usepackage{array}
\usepackage{float}
\usepackage{units}
\usepackage{mathrsfs}
\usepackage{enumitem}
\usepackage{amstext}
\usepackage{amsthm}
\usepackage{amssymb}
\usepackage{graphicx}
\usepackage[unicode=true,pdfusetitle,
 bookmarks=true,bookmarksnumbered=false,bookmarksopen=false,
 breaklinks=false,pdfborder={0 0 1},backref=false,colorlinks=false]
 {hyperref}
\hypersetup{
 colorlinks=true,citecolor=blue,linkcolor=blue,linktocpage=true}

\makeatletter

\newcommand*\LyXbar{\rule[0.585ex]{1.2em}{0.25pt}}
\providecommand{\tabularnewline}{\\}

\numberwithin{equation}{section}
\numberwithin{figure}{section}
\numberwithin{table}{section}
\theoremstyle{plain}
\newtheorem{thm}{\protect\theoremname}[section]
  \theoremstyle{definition}
  \newtheorem{defn}[thm]{\protect\definitionname}
  \theoremstyle{plain}
  \newtheorem{lem}[thm]{\protect\lemmaname}
  \theoremstyle{plain}
  \newtheorem*{lem*}{\protect\lemmaname}
  \theoremstyle{remark}
  \newtheorem*{note*}{\protect\notename}
  \theoremstyle{remark}
  \newtheorem{rem}[thm]{\protect\remarkname}
  \theoremstyle{definition}
  \newtheorem{example}[thm]{\protect\examplename}
  \theoremstyle{plain}
  \newtheorem{cor}[thm]{\protect\corollaryname}
  \theoremstyle{remark}
  \newtheorem*{acknowledgement*}{\protect\acknowledgementname}

\@ifundefined{date}{}{\date{}}
\allowdisplaybreaks
\usepackage{needspace}
\usepackage{refstyle}
\usepackage{enumitem}
\usepackage[skins,breakable]{tcolorbox}

\newref{lem}{refcmd={Lemma \ref{#1}}}
\newref{thm}{refcmd={Theorem \ref{#1}}}
\newref{cor}{refcmd={Corollary \ref{#1}}}
\newref{sec}{refcmd={Section \ref{#1}}}
\newref{sub}{refcmd={Section \ref{#1}}}
\newref{subsec}{refcmd={Section \ref{#1}}}
\newref{chap}{refcmd={Chapter \ref{#1}}}
\newref{prop}{refcmd={Proposition \ref{#1}}}
\newref{exa}{refcmd={Example \ref{#1}}}
\newref{tab}{refcmd={Table \ref{#1}}}
\newref{rem}{refcmd={Remark \ref{#1}}}
\newref{def}{refcmd={Definition \ref{#1}}}
\newref{fig}{refcmd={Figure \ref{#1}}}

\setlist[enumerate]{itemsep=5pt,topsep=3pt}
\setlist[enumerate,1]{label=(\roman*),ref=\roman*}
\setlist[enumerate,2]{label=(\alph*),ref=\theenumi \alph*}

\usepackage{scalerel}

\newcommand\reallywidehat[1]{\arraycolsep=0pt\relax%
\begin{array}{c}
\stretchto{
  \scaleto{
    \scalerel*[\widthof{\ensuremath{#1}}]{\kern-.5pt\bigwedge\kern-.5pt}
    {\rule[-\textheight/2]{1ex}{\textheight}} 
  }{\textheight} %
}{0.5ex}\\           
#1\\                 
\rule{-1ex}{0ex}
\end{array}
}

\AtBeginDocument{
  
}

\makeatother

  \providecommand{\acknowledgementname}{Acknowledgement}
  \providecommand{\corollaryname}{Corollary}
  \providecommand{\definitionname}{Definition}
  \providecommand{\examplename}{Example}
  \providecommand{\lemmaname}{Lemma}
  \providecommand{\notename}{Note}
  \providecommand{\remarkname}{Remark}
\providecommand{\theoremname}{Theorem}

\begin{document}

\title{Spectral pairs and positive definite tempered distributions}

\author{Palle Jorgensen}

\address{(Palle E.T. Jorgensen) Department of Mathematics, The University
of Iowa, Iowa City, IA 52242-1419, U.S.A. }

\email{palle-jorgensen@uiowa.edu}

\urladdr{http://www.math.uiowa.edu/\textasciitilde{}jorgen/}

\author{Feng Tian}

\address{(Feng Tian) Department of Mathematics, Hampton University, Hampton,
VA 23668, U.S.A.}

\email{feng.tian@hamptonu.edu}


\begin{abstract}
The present paper presents two new approaches to Fourier series and
spectral analysis of singular measures.
\end{abstract}

\maketitle
\tableofcontents{}
\begin{quote}
\textquotedblleft There cannot be a language more universal and more
simple, more free from errors and obscurities...more worthy to express
the invariable relations of all natural things {[}than mathematics{]}.
{[}It interprets{]} all phenomena by the same language, as if to attest
the unity and simplicity of the plan of the universe, and to make
still more evident that unchangeable order which presides over all
natural causes\textquotedblright{} \LyXbar{} Joseph Fourier, The Analytical
Theory of Heat\vspace{1cm}
\end{quote}
A spectral pair in $\mathbb{R}^{k}$, is a pair $\left(\nu,\Lambda\right)$
where $\nu$ is a finite positive Borel measure on $\mathbb{R}^{k}$,
and where $\Lambda$ is a strictly discrete subset of $\mathbb{R}^{k}$
(see \cite{MR1215311,MR1655831,2017arXiv170504198H,MR1629847,MR1655832,MR1785282,MR2921088}).
From a given set $\Lambda$, there is then an associated positive
definite tempered distribution $F_{\Lambda}$, and a corresponding
generalized reproducing kernel Hilbert space $\mathscr{H}_{\Lambda}$.
We say that $\left(\nu,\Lambda\right)$ is a spectral pair iff the
set of Fourier exponentials from $\Lambda$ forms an orthogonal basis
in $L^{2}\left(\nu\right)$; intuitively, $L^{2}\left(\nu\right)$
admits an orthogonal $\Lambda$-Fourier series representation. In
this paper, we study the occurrence of spectral pairs within the framework
of positive definite tempered distributions, as defined by Laurent
Schwartz \cite{MR0179587}. Given $\Lambda$, we present a necessary
and sufficient condition for the existence of a finite positive Borel
measure $\nu$ such that $\left(\nu,\Lambda\right)$ is a spectral
pair. Our method relies on use of the generalized reproducing kernel
Hilbert space $\mathscr{H}_{\Lambda}$.

\section{Preliminaries}

In our theorems and proofs, we shall make use the particular reproducing
kernel Hilbert spaces (RKHSs) which allow us to give explicit formulas
for our solutions. The general framework of RKHSs were pioneered by
Aronszajn in the 1950s \cite{MR0051437}; and subsequently they have
been used in a host of applications; e.g., \cite{MR2327597,MR2558684}.

Positive definite functions on groups $G$ yield a special class of
positive definite kernels. In general, the covariance kernel of a
Gaussian process, indexed by $G$, is a positive definite kernel.
If the $G$-process is $G$-stationary, then its covariance kernel
will be defined from a positive definite function. While these notions
make sense in general, we shall restrict attention here to the case
of locally compact abelian groups, with an emphasis on the case $G=\mathbb{R}^{k}$.
But to simplify notation, it is convenient to begin the discussion
with the case $k=1$. The extension to the cases $k>1$ is fairly
straightforward, and we shall concentrate on the general case $G=\mathbb{R}^{k}$
in the main part of the paper. Our analysis will go beyond the case
of positive definite functions. Indeed, our main theorem is stated
in the context of positive definite tempered distributions. A comparison
of the analysis in the two cases will be given in section \ref{sec:para}
below.

\subsubsection{The RKHS $\mathscr{H}_{F}$}

For simplicity we focus on the case $G=\mathbb{R}$. 
\begin{defn}
\label{def:pdf}Let $\Omega$ be an open domain in $\mathbb{R}$.
A function $F:\Omega-\Omega\rightarrow\mathbb{C}$ is \emph{positive
definite } if 
\begin{equation}
\sum\nolimits _{i}\sum\nolimits _{j}c_{i}\overline{c}_{j}F\left(x_{i}-x_{j}\right)\geq0\label{eq:def-pd}
\end{equation}
for all finite sums with $c_{i}\in\mathbb{C}$, and all $x_{i}\in\Omega$.
We assume that all the p.d. functions are continuous and bounded.
\end{defn}
\begin{lem}[Two equivalent conditions for p.d.]
If $F$ is given continuous on $\mathbb{R}$, we have the following
two equivalent conditions for the positive definite property:
\begin{enumerate}
\item $\forall n\in\mathbb{N}$, $\forall\left\{ x_{i}\right\} _{1}^{n}$,
$\forall\left\{ c_{i}\right\} _{1}^{n}$, $x_{i}\in\mathbb{R}$, $c_{i}\in\mathbb{C}$,
\[
\sum_{i}\sum_{j}c_{i}\overline{c}_{j}F\left(x_{i}-x_{j}\right)\geq0;
\]
\item $\forall\varphi\in C_{c}\left(\mathbb{R}\right)$, we have: 
\[
\int_{\mathbb{R}}\int_{\mathbb{R}}\varphi\left(x\right)\overline{\varphi\left(y\right)}F\left(x-y\right)dxdy\geq0.
\]
\end{enumerate}
\end{lem}
\begin{proof}
Use Riemann integral approximation, and note that $F\left(\cdot-x\right)\in\mathscr{H}_{F}$,
and $\varphi\ast F\in\mathscr{H}_{F}$. (See details below.)
\end{proof}

Consider a continuous positive definite function so $F$ is defined
on $\Omega-\Omega$. Set 
\begin{equation}
F_{y}\left(x\right):=F\left(x-y\right),\;\forall x,y\in\Omega.\label{eq:H0}
\end{equation}
Let $\mathscr{H}_{F}$ be the \emph{reproducing kernel Hilbert space
}(RKHS), which is the completion of 
\begin{equation}
\left\{ \sum\nolimits _{\text{finite}}c_{j}F_{x_{j}}\mid x_{j}\in\Omega,\:c_{j}\in\mathbb{C}\right\} \label{eq:H1}
\end{equation}
with respect to the inner product
\begin{equation}
\left\langle \sum\nolimits _{i}c_{i}F_{x_{i}},\sum\nolimits _{j}d_{j}F_{y_{j}}\right\rangle _{\mathscr{H}_{F}}:=\sum\nolimits _{i}\sum\nolimits _{j}c_{i}\overline{d}_{j}F\left(x_{i}-y_{j}\right);\label{eq:ip-discrete}
\end{equation}
modulo the subspace of functions of $\left\Vert \cdot\right\Vert _{\mathscr{H}_{F}}$-norm
zero. 

Below, we introduce an equivalent characterization of the RKHS $\mathscr{H}_{F}$,
which we will be working with in the rest of the paper. 
\begin{lem}
\label{lem:dense}Fix $\Omega=(0,\alpha)$. Let $\varphi_{n,x}\left(t\right)=n\varphi\left(n\left(t-x\right)\right)$,
for all $t\in\Omega$; where $\varphi$ satisfies 

\begin{enumerate}
\item $\mathrm{supp}\left(\varphi\right)\subset\left(-\alpha,\alpha\right)$;
\item $\varphi\in C_{c}^{\infty}$, $\varphi\geq0$;
\item $\int\varphi\left(t\right)dt=1$. Note that $\lim_{n\rightarrow\infty}\varphi_{n,x}=\delta_{x}$,
the Dirac measure at $x$.
\end{enumerate}
\end{lem}

\begin{lem}
\label{lem:RKHS-def-by-integral}The RKHS, $\mathscr{H}_{F}$, is
the Hilbert completion of the functions 
\begin{equation}
F_{\varphi}\left(x\right)=\int_{\Omega}\varphi\left(y\right)F\left(x-y\right)dy,\;\forall\varphi\in C_{c}^{\infty}\left(\Omega\right),x\in\Omega\label{eq:H2}
\end{equation}
with respect to the inner product
\begin{equation}
\left\langle F_{\varphi},F_{\psi}\right\rangle _{\mathscr{H}_{F}}=\int_{\Omega}\int_{\Omega}\varphi\left(x\right)\overline{\psi\left(y\right)}F\left(x-y\right)dxdy,\;\forall\varphi,\psi\in C_{c}^{\infty}\left(\Omega\right).\label{eq:hi2}
\end{equation}
In particular, 
\begin{equation}
\left\Vert F_{\varphi}\right\Vert _{\mathscr{H}_{F}}^{2}=\int_{\Omega}\int_{\Omega}\varphi\left(x\right)\overline{\varphi\left(y\right)}F\left(x-y\right)dxdy,\;\forall\varphi\in C_{c}^{\infty}\left(\Omega\right)\label{eq:hn2}
\end{equation}
and 
\begin{equation}
\left\langle F_{\varphi},F_{\psi}\right\rangle _{\mathscr{H}_{F}}=\int_{\Omega}F_{\varphi}\left(x\right)\overline{\psi\left(x\right)}dx,\;\forall\phi,\psi\in C_{c}^{\infty}(\Omega).
\end{equation}
\end{lem}
\begin{proof}
Indeed, by Lemma \ref{lem:RKHS-def-by-integral}, we have 
\begin{equation}
\left\Vert F_{\varphi_{n,x}}-F\left(\cdot-x\right)\right\Vert _{\mathscr{H}_{F}}\rightarrow0,\;\mbox{as }n\rightarrow\infty.\label{eq:approx}
\end{equation}
Hence $\left\{ F_{\varphi}\right\} _{\varphi\in C_{c}^{\infty}\left(\Omega\right)}$
spans a dense subspace in $\mathscr{H}_{F}$. 

For more details, see \cite{MR863534,MR874059}.

\end{proof}

The following two conditions (\ref{eq:bdd})($\Leftrightarrow$(\ref{eq:bdd2}))
below will be used to characterize elements in the Hilbert space $\mathscr{H}_{F}$.
\begin{thm}
\label{thm:HF}A continuous function $\xi:\Omega\rightarrow\mathbb{C}$
is in $\mathscr{H}_{F}$ if and only if there exists $A_{0}>0$, such
that
\begin{equation}
\sum\nolimits _{i}\sum\nolimits _{j}c_{i}\overline{c}_{j}\xi\left(x_{i}\right)\overline{\xi\left(x_{j}\right)}\leq A_{0}\sum\nolimits _{i}\sum\nolimits _{j}c_{i}\overline{c}_{j}F\left(x_{i}-x_{j}\right)\label{eq:bdd}
\end{equation}
for all finite system $\left\{ c_{i}\right\} \subset\mathbb{C}$ and
$\left\{ x_{i}\right\} \subset\Omega$.

Equivalently, for all $\psi\in C_{c}^{\infty}\left(\Omega\right)$,
\begin{eqnarray}
\left|\int_{\Omega}\psi\left(y\right)\overline{\xi\left(y\right)}dy\right|^{2} & \leq & A_{0}\int_{\Omega}\int_{\Omega}\psi\left(x\right)\overline{\psi\left(y\right)}F\left(x-y\right)dxdy\label{eq:bdd2}
\end{eqnarray}
Note that, if $\xi\in\mathscr{H}_{F}$, then the LHS of (\ref{eq:bdd2})
is $\vert\left\langle F_{\psi},\xi\right\rangle _{\mathscr{H}_{F}}\vert^{2}$.
Indeed,
\begin{eqnarray*}
\left|\left\langle \xi,F_{\psi}\right\rangle _{\mathscr{H}_{F}}\right|^{2} & = & \left|\left\langle \xi,\int_{\Omega}\psi\left(y\right)F_{y}\:dy\right\rangle _{\mathscr{H}_{F}}\right|^{2}\\
 & = & \left|\int_{\Omega}\overline{\psi\left(y\right)}\left\langle \xi,F_{y}\right\rangle _{\mathscr{H}_{F}}dy\right|^{2}\\
 & = & \left|\int_{\Omega}\overline{\psi\left(y\right)}\xi\left(y\right)dy\right|^{2}\;(\text{by the reproducing property}).
\end{eqnarray*}
\end{thm}

\section{\label{sec:para}The parallels of p.d. functions vs distributions}

The study of positive definite (p.d.) functions, and p.d. kernels,
is motivated by diverse themes in analysis and operator theory, in
white noise analysis, applications of reproducing kernel (RKHS) theory,
extensions by Laurent Schwartz, and in reflection positivity from
quantum physics (see the cited references.) The parallels between
Bochner's theorem (for continuous p.d. functions), and the generalization
to Bochner/Schwartz representations for positive definite tempered
distributions will be made clear. In the first case, we have the Bochner
representation via finite positive measures $\mu$; and in the second
case, instead via tempered positive measures. This parallel also helps
make precise the respective reproducing kernel Hilbert spaces (RKHSs).
This further leads to a more unified approach to the treatment of
the stationary-increment Gaussian processes \cite{MR2793121,MR2966130,MR3402823}.
A key argument will rely on the existence of a unitary representation
$U$ of $\left(\mathbb{R},+\right)$, acting on the particular RKHS
under discussion. In fact, the same idea (with suitable modifications)
will also work in the wider context of locally compact groups. In
the abelian case, we shall make use of the Stone representation for
$U$ in the form of orthogonal projection valued measures; and in
more general settings, the Stone-Naimark-Ambrose-Godement (SNAG) representation
\cite{MR1503079}. 
\begin{thm}
\label{thm:p1}~
\begin{enumerate}
\item[(a)]  Let $F$ be a continuous positive definite (p.d.) function on $\mathbb{R}$
(a p.d. tempered distribution \cite{MR0185423,MR0179587}); then there
is a unique finite positive Borel measure $\mu$ on $\mathbb{R}$
(resp., a unique tempered measure on $\mathbb{R}$) such that $F=\widehat{\mu}$.
\item[(b)]  Given $F$ as above, let $\mathscr{H}_{F}$ denote the corresponding
kernel Hilbert space, i.e., the Hilbert completion of $\left\{ \varphi\ast F\right\} _{\varphi\in C_{c}\left(\Omega\right)}$
(resp. $\varphi\in\mathcal{S}$) w.r.t 
\[
\left\Vert \varphi\ast F\right\Vert _{\mathscr{H}_{F}}^{2}=\int_{\mathbb{R}}\int_{\mathbb{R}}\varphi\left(x\right)\overline{\varphi\left(y\right)}F\left(x-y\right)dxdy
\]
resp., $\left\langle F\left(x-y\right),\varphi\ast\overline{\varphi}\right\rangle $;
action in the sense of distributions. Then there is a unique isometric
transform
\begin{gather*}
\mathscr{H}_{F}\xrightarrow{\;T_{F}\;}L^{2}\left(\mathbb{R},\mathscr{B},\mu\right),\quad T_{F}\left(\varphi\ast F\right)=\widehat{\varphi},\;\text{i.e.,}\\
\left\Vert \varphi\ast F\right\Vert _{\mathscr{H}_{F}}^{2}=\int_{\mathbb{R}}\left|\widehat{\varphi}\right|^{2}d\mu=\left\Vert T_{F}\varphi\right\Vert _{L^{2}\left(\mu\right)}^{2}.
\end{gather*}
\item[(c)]  If $\mu$ is tempered, e.g., i.e., $\int_{\mathbb{R}}\frac{d\mu\left(\lambda\right)}{1+\lambda^{2}}<\infty$,
then 
\[
\left\Vert \varphi\ast F\right\Vert _{\mathscr{H}_{F}}^{2}=\int\left(\left|\widehat{\varphi}\right|^{2}+\big|\widehat{\left(D_{x}\varphi\right)}\big|^{2}\right)\frac{d\mu\left(\lambda\right)}{1+\lambda^{2}};
\]
where $D_{x}\varphi=\frac{d\varphi}{dx}$, and where ``$\widehat{\,\cdot\,}$''
denotes the standard Fourier transform on $\mathbb{R}$. 
\end{enumerate}
\end{thm}
\begin{proof}
See \cite{JT17a}.
\end{proof}

\subsection*{OVERVIEW}

~

\noindent\begin{minipage}[t]{1\columnwidth}%
\begin{minipage}[t]{0.48\columnwidth}%
\textbf{Continuous p.d. functions on $\mathbb{R}$}
\begin{lem*}
Let $F$ be a continuous function on $\mathbb{R}$. Then the following
are equivalent:
\begin{enumerate}[leftmargin=18pt]
\item $F$ is p.d., i.e., $\forall\varphi\in C_{c}\left(\mathbb{R}\right)$,
we have 
\begin{equation}
\int_{\mathbb{R}}\int_{\mathbb{R}}\varphi\left(x\right)\overline{\varphi\left(y\right)}F\left(x-y\right)dxdy\geq0.
\end{equation}
\item $\forall\left\{ x_{j}\right\} _{j=1}^{n}\subset\mathbb{R}$, $\forall\left\{ c_{j}\right\} _{j=1}^{n}\subset\mathbb{C}$,
and $\forall n\in\mathbb{N}$, we have 
\begin{equation}
\sum_{j=1}^{n}\sum_{k=1}^{n}c_{j}\overline{c}_{k}F\left(x_{j}-x_{k}\right)\geq0.
\end{equation}
 
\end{enumerate}
\end{lem*}
\end{minipage}\hfill{}%
\begin{minipage}[t]{0.48\columnwidth}%
\textbf{p.d. tempered distributions on $\mathbb{R}$}
\begin{lem*}
Let $F$ be a tempered distribution on $\mathbb{R}$. Then $F$ is
p.d. if and only if 
\begin{equation}
\int_{\mathbb{R}}\int_{\mathbb{R}}\varphi\left(x\right)\overline{\varphi\left(y\right)}F\left(x-y\right)dxdy\geq0
\end{equation}
hold, for all $\varphi\in\mathcal{S}$, where $\mathcal{S}$ is the
\textup{Schwartz space}. 

Equivalently, 
\begin{equation}
\left\langle F\left(x-y\right),\varphi\otimes\overline{\varphi}\right\rangle \geq0,\;\forall\varphi\in\mathcal{S}.
\end{equation}
Here $\left\langle \cdot,\cdot\right\rangle $ denotes distribution
action. 
\end{lem*}
\end{minipage}%
\end{minipage}
\begin{center}
\textbf{}%
\noindent\begin{minipage}[t]{1\columnwidth}%
\begin{center}
\textbf{RKHS}
\par\end{center}
\begin{minipage}[t]{0.48\columnwidth}%
\textbf{Bochner's theorem.} 

$\exists!$ positive finite measure $\mu$ on $\mathbb{R}$ such that
\[
F\left(x\right)=\int_{\mathbb{R}}e^{ix\lambda}d\mu\left(\lambda\right).
\]
\end{minipage}\hfill{}%
\begin{minipage}[t]{0.48\columnwidth}%
\textbf{Bochner/Schwartz}

$\exists$ positive tempered measure $\mu$ on $\mathbb{R}$ such
that 
\[
F=\widehat{\mu}
\]
where $\widehat{\mu}$ is in the sense of distribution. %
\end{minipage}%
\end{minipage}
\par\end{center}

\begin{center}
\textbf{}%
\noindent\begin{minipage}[t]{1\columnwidth}%
\begin{minipage}[t]{0.48\columnwidth}%
Let $\mathscr{H}_{F}$ be the RKHS of $F$. 
\begin{itemize}[leftmargin=10pt]
\item Then 
\begin{equation}
\left\Vert \varphi\ast F\right\Vert _{\mathscr{H}_{F}}^{2}=\int_{\mathbb{R}}\left|\widehat{\varphi}\left(\lambda\right)\right|^{2}d\mu\left(\lambda\right)
\end{equation}
where $\widehat{\varphi}=$ the Fourier transform. 
\item $F$ admits the factorization 
\[
F\left(x_{1}-x_{2}\right)=\left\langle F\left(\cdot-x_{1}\right),F\left(\cdot-x_{2}\right)\right\rangle _{\mathscr{H}_{F}}
\]
$\forall x_{1},x_{2}\in\mathbb{R}$, with 

$\mathbb{R}\ni x\longrightarrow F\left(\cdot-x\right)\in\mathscr{H}_{F}$. 
\end{itemize}
\end{minipage}\hfill{}%
\begin{minipage}[t]{0.48\columnwidth}%
Let $\mathscr{H}_{F}$ denote the corresponding RKHS. 
\begin{itemize}[leftmargin=10pt]
\item For all $\varphi\in\mathcal{S}$, we have 
\begin{equation}
\left\Vert \varphi\ast F\right\Vert _{\mathscr{H}_{F}}^{2}=\left\langle F\left(x-y\right),\varphi\otimes\overline{\varphi}\right\rangle ,
\end{equation}
distribution action. 
\item $\mathcal{S}\ni\varphi\longmapsto\varphi\ast F\in\mathscr{H}_{F}$,
where 
\[
\left(\varphi\ast F\right)\left(\cdot\right)=\int\varphi\left(y\right)F\left(\cdot-y\right)dy.
\]
\end{itemize}
\end{minipage}%
\end{minipage}
\par\end{center}

\begin{center}
\textbf{}%
\noindent\begin{minipage}[t]{1\columnwidth}%
\begin{center}
\textbf{Applications}
\par\end{center}
\begin{minipage}[t]{0.48\columnwidth}%
Now applied to Bochner's theorem. 

Set $\mathscr{H}_{F}=$ RKHS of $F$, and $w_{0}=F\left(\cdot-0\right)$.
Then 
\[
U_{t}w_{0}=w_{t}=F\left(\cdot-t\right),\;t\in\mathbb{R}
\]
defines a strongly continuous unitary representation of $\mathbb{R}$. %
\end{minipage}\hfill{}%
\begin{minipage}[t]{0.48\columnwidth}%
On white noise space:
\[
\mathbb{E}\left(e^{i\left\langle \varphi,\cdot\right\rangle }\right)=e^{-\frac{1}{2}\int\left|\widehat{\varphi}\right|^{2}d\mu}
\]
where $\mathbb{E}\left(\cdots\right)=$ expectation w.r.t the Gaussian
path-space measure. 

(The proof for the special case when $F$ is assumed p.d. and continuous
carries over with some changes to the case when $F$ is a p.d. tempered
distribution.)%
\end{minipage}%
\end{minipage}
\par\end{center}
\begin{note*}
In both cases, we have the following representation for vectors in
the RKHS $\mathscr{H}_{F}$: 
\begin{equation}
\left\langle \varphi\ast F,\psi\ast F\right\rangle _{\mathscr{H}_{F}}=\left\langle \varphi\ast\overline{\psi},F\right\rangle ,\;\forall\varphi,\psi\in\mathcal{S};
\end{equation}
where $\varphi\ast F:=$ the standard convolution w.r.t. Lebesgue
measure.
\end{note*}

\section{Spectral pairs and p.d. distributions}
\begin{defn}[L. Schwartz]
Introduce the following locally convex topological (LCT) spaces,
such that we have continuous inclusions
\[
C_{c}^{\infty}\left(\mathbb{R}^{k}\right)\hookrightarrow\mathcal{S}_{k}\hookrightarrow C^{\infty}\left(\mathbb{R}^{k}\right),
\]
with corresponding duals for the distribution spaces, 
\[
\left(C_{c}^{\infty}\left(\mathbb{R}^{k}\right)\right)'\hookleftarrow\mathcal{S}_{k}'\hookleftarrow\left(C^{\infty}\left(\mathbb{R}^{k}\right)\right)'
\]
where $\mathscr{D}':=C_{c}^{\infty}\left(\mathbb{R}^{k}\right)'=$
all distributions, and $\mathscr{E}':=\left(C^{\infty}\right)'=$
all distributions of compact support. 
\end{defn}
\begin{defn}
\label{def:ud}A subset $\Lambda\subset\mathbb{R}^{k}$, $k\in\mathbb{N}$,
is said to be \emph{uniformly discrete} iff (Def.) there exists $r>0$
such that $\left|\lambda-\lambda'\right|\geq r$, for all $\lambda,\lambda'\in\Lambda$,
$\lambda\neq\lambda'$. In this case, we say $\Lambda\in\mathscr{D}_{k}$. 
\end{defn}
\begin{defn}[\cite{MR1215311,MR1655831,2017arXiv170504198H}]
~
\begin{enumerate}[label=(\alph{enumi})]
\item Let $\nu$ be a finite Borel measure on $\mathbb{R}^{k}$, assumed
normalized for simplicity, then a subset $\Lambda\subset\mathbb{R}^{k}$
is said to be a spectrum\emph{ }iff $\left\{ e^{i\lambda x}\right\} _{\lambda\in\Lambda}$
is an orthonormal basis (ONB) in $L^{2}\left(\nu\right)$. In this
case, we say that $\left(\nu,\Lambda\right)$ is a \emph{spectral
pair}.
\item If instead 
\[
\int_{\Omega}\left|f\left(x\right)\right|^{2}d\nu\left(x\right)=\sum_{\lambda\in\Lambda}\left|\int_{\mathbb{R}}f\left(x\right)e^{-i\lambda x}d\nu\left(x\right)\right|^{2}
\]
for all $f\in L^{2}\left(\nu\right)$, then we say that $\left(\nu,\Lambda\right)$
is a \emph{Parseval spectral pair}. Clearly a spectral pair is also
a Parseval spectral pair.
\item Given $\nu$, set 
\begin{equation}
\mathscr{L}\left(\nu\right):=\left\{ \Lambda\subset\mathbb{R}^{k}\mathrel{;}\left(\nu,\Lambda\right)\text{ is a spectral pair}\right\} ;\label{eq:d9}
\end{equation}
and given $\Lambda\in\mathscr{D}_{k}$, set 
\begin{equation}
\mathscr{M}\left(\Lambda\right):=\left\{ \nu\in\text{Prob}_{\mathscr{B}}\left(\mathbb{R}^{k}\right)\mathrel{;}\left(\nu,\Lambda\right)\:\text{is a spectral pair}\right\} .\label{eq:d10}
\end{equation}
\item If $\left(\nu,\Lambda\right)$ is a spectral pair, we say that $\nu\in\mathscr{M}\left(\Lambda\right)$,
and $\Lambda\in\mathscr{L}\left(\nu\right)$. If $\left(\nu,\Lambda\right)$
is a Parseval spectral pair, we say that $\Lambda\in\mathscr{L}_{\text{Pa}}\left(\nu\right)$,
and $\nu\in\mathscr{M}_{\text{Pa}}\left(\Lambda\right)$. 
\end{enumerate}
\end{defn}

\begin{rem}
For details on Parseval frames and their recent applications, we refer
to \cite{MR3048586,MR3318644,MR3421919,MR3606269,MR3702855,MR2735759,MR3329098,MR3622655}.

The case when $\nu\in\text{Prob}_{\mathscr{B}}\left(\mathbb{R}^{k}\right)$
is an IFS-Cantor measure has been studied extensively in the literature.
See, e.g, \cite{MR1215311,MR1655831,MR2509326,MR3055992,MR1008470,MR951745,MR1083085,MR2457304,MR2945156}.
\end{rem}
\begin{defn}[L. Schwartz \cite{MR0179587}]
\label{def:pd}A tempered distribution $F\in\mathcal{S}'$ is \emph{positive
definite} (p.d.) on $\mathbb{R}^{k}$ if 
\begin{equation}
\left\langle F\left(x-y\right),\varphi\otimes\overline{\varphi}\right\rangle \geq0,\;\forall\varphi\in\mathcal{S}_{k}.\label{eq:d1}
\end{equation}
where $\mathcal{S}_{k}:=$ the Schwartz space on $\mathbb{R}^{k}$. 
\end{defn}
\begin{defn}
A positive measure $\mu$ on $\mathbb{R}^{k}$ is said to be \emph{tempered}
iff (Def.) $\exists M\in\mathbb{N}$ such that 
\begin{equation}
\int_{\mathbb{R}^{k}}\frac{d\mu\left(\lambda\right)}{1+\left\Vert \lambda\right\Vert ^{2M}}<\infty.\label{eq:d2}
\end{equation}
\end{defn}
\begin{thm}[Schwartz \cite{MR0179587}]
\label{thm:d7} A tempered distribution $F$ is positive definite
if and only if there exists a positive tempered Borel measure $\mu$
on $\mathbb{R}^{k}$ such that $F=\widehat{\mu}$, more precisely,
\begin{equation}
\int_{\mathbb{R}^{k}}F\left(x\right)\varphi\left(x\right)dx=\int_{\mathbb{R}^{k}}\widehat{\varphi}\left(\lambda\right)d\mu\left(\lambda\right),\;\forall\varphi\in\mathcal{S}_{k}.\label{eq:d3}
\end{equation}
In particular (see (\ref{eq:d1})) the measure $\mu$ will satisfy:
\begin{equation}
\int_{\mathbb{R}^{k}}\int_{\mathbb{R}^{k}}F\left(x-y\right)\varphi\left(x\right)\overline{\varphi\left(y\right)}dxdy=\int_{\mathbb{R}^{k}}\left|\widehat{\varphi}\left(\lambda\right)\right|^{2}d\mu\left(\lambda\right)\label{eq:d4}
\end{equation}
with a slight abuse of notation. Here, ``$\widehat{\,\cdot\,}$''
denotes the standard Fourier transform. 
\end{thm}
\begin{defn}
Let $F$ be a positive definite tempered distribution, then on the
functions 
\begin{equation}
\left(\varphi\ast F\right)\left(x\right)=\int_{\mathbb{R}^{k}}\varphi\left(y\right)F\left(x-y\right)dy,\;\varphi\in\mathcal{S}_{k},\label{eq:d5}
\end{equation}
set 
\begin{equation}
\left\langle \varphi\ast F,\psi\ast F\right\rangle _{\mathscr{H}_{F}}:=\int_{\mathbb{R}^{k}}\int_{\mathbb{R}^{k}}\varphi\left(x\right)\overline{\psi\left(y\right)}F\left(x-y\right)dxdy,\;\forall\varphi,\psi\in\mathcal{S}_{k}.\label{eq:d6}
\end{equation}
Then $\left\{ \varphi\ast F\mathrel{;}\varphi\in\mathcal{S}_{k}\right\} $
forms a pre-Hilbert space with respect to the norm specified in (\ref{eq:d6})
and (\ref{eq:d8}). Set 
\begin{equation}
\mathscr{H}_{F}:=\text{the Hilbert completion from }\left(\ref{eq:d5}\right)\:\&\:\left(\ref{eq:d6}\right).\label{eq:d7}
\end{equation}
\end{defn}
It follows that
\begin{lem}
A function $\varphi\ast F$ is in $\mathscr{H}_{F}$ if and only if
$\widehat{\varphi}\in L^{2}\left(\mu\right)$, and then 
\begin{equation}
\left\Vert \varphi\ast F\right\Vert _{\mathscr{H}_{F}}^{2}=\int_{\mathbb{R}^{k}}\left|\widehat{\varphi}\left(\lambda\right)\right|^{2}d\mu\left(\lambda\right)\label{eq:d8}
\end{equation}
where $\mu$ is the tempered measure from (\ref{eq:d3}).
\end{lem}
\begin{lem}
If $\Lambda\in\mathscr{D}_{k}$, then 
\begin{equation}
F_{\Lambda}\left(x\right):=\sum_{\lambda\in\Lambda}e^{i\lambda x}\label{eq:d11}
\end{equation}
is a positive definite (p.d.) tempered distribution. 
\end{lem}
\begin{proof}
Let $\varphi\in\mathcal{S}_{k}$, and let $r>0$ be as in Definition
\ref{def:ud}. Then 
\[
\left\langle F_{\Lambda},\varphi\right\rangle =\int_{\mathbb{R}^{k}}\sum\nolimits _{\lambda\in\Lambda}e^{i\lambda x}\varphi\left(x\right)dx
\]
where $dx=dx_{1}\cdots dx_{k}$ is the usual $k$-Lebesgue measure
on $\mathbb{R}^{k}$; and so 
\begin{equation}
\left\langle F_{\Lambda},\varphi\right\rangle =\sum_{\lambda\in\Lambda}\widehat{\varphi}\left(\lambda\right).\label{eq:dd1}
\end{equation}
Set $\left|\lambda\right|=\left(\lambda_{1}^{2}+\cdots+\lambda_{k}^{2}\right)^{\frac{1}{2}}$.
From the assumption in the lemma; see Definition \ref{def:ud}, it
follows that there is an $M\in\mathbb{N}$, depending on $k$ and
$r$ such that 
\begin{equation}
C_{M}\left(\Lambda\right):=\sum_{\lambda\in\Lambda}\frac{1}{1+\left|\lambda\right|^{2M}}<\infty.\label{eq:dd2}
\end{equation}
Let $\Delta_{k}$ be the usual Laplacian on $\mathbb{R}^{k}$, then
for $\varphi\in\mathcal{S}_{k}$ (the space of Schwartz test functions),
we get 
\begin{equation}
\widehat{\varphi}\left(\lambda\right)=\left(\left(I-\Delta_{k}^{M}\right)\varphi\right)^{\wedge}\left(\lambda\right)\big/(1+\left|\lambda\right|^{2M}),\;\forall\lambda\in\mathbb{R}^{k}.\label{eq:dd3}
\end{equation}
Again, using $\varphi\in\mathcal{S}_{k}$, 
\begin{equation}
A\left(\varphi\right):=\sup_{\lambda\in\mathbb{R}^{k}}\left|\left(\left(I-\Delta_{k}^{M}\right)\varphi\right)^{\wedge}\left(\lambda\right)\right|<\infty.\label{eq:dd4}
\end{equation}
Combining the estimates (\ref{eq:dd2}) \& (\ref{eq:dd4}), we get
\[
\left|\left\langle F_{\Lambda},\varphi\right\rangle \right|\leq A\left(\varphi\right)C_{M}\left(\Lambda\right).
\]
Since $A\left(\varphi\right)$ is one of the seminorms given in the
definition of the topology on Schwartz' space of tempered test functions
$\mathcal{S}_{k}$, the desired conclusion follows. 
\end{proof}
\begin{example}
Let $\nu=\nu_{4}$ be the $\nicefrac{1}{4}$-Cantor measure, i.e.,
the unique solution to the IFS-equation in $\text{Prob}_{\mathscr{B}}\left(\mathbb{R}\right)$
\begin{equation}
\frac{1}{2}\int_{\mathbb{R}}\left(f\left(\frac{x}{4}\right)+f\left(\frac{x+2}{4}\right)\right)d\nu_{4}\left(x\right)=\int_{\mathbb{R}}f\left(x\right)d\nu_{4}\left(x\right),\label{eq:d12}
\end{equation}
and let 
\begin{align}
\Lambda_{4} & =\left\{ 0,1,4,5,16,17,20,21,64,65,\cdots\right\} \label{eq:d13}\\
 & =\left\{ \sum\nolimits _{0}^{\text{finite}}b_{i}4^{i}\mathrel{;}b_{i}\in\left\{ 0,1\right\} \right\} .\nonumber 
\end{align}
Then $\left(\nu_{4},\Lambda_{4}\right)$ is a spectral pair; and so
$\Lambda_{4}\in\mathscr{L}\left(\nu_{4}\right)$, and $\nu_{4}\in\mathscr{M}\left(\Lambda_{4}\right)$.
See Figure \ref{fig:sm}.

It is known that $\left(\nu_{3},\Lambda_{3}\right)$ is \emph{not}
a spectral pair, where $\nu_{3}$ is the unique solution in $\text{Prob}_{\mathscr{B}}\left(\mathbb{R}\right)$
to 
\begin{equation}
\frac{1}{2}\int_{\mathbb{R}}\left(f\left(\frac{x}{3}\right)+f\left(\frac{x+2}{3}\right)\right)d\nu_{3}\left(x\right)=\int_{\mathbb{R}}f\left(x\right)d\nu_{3}\left(x\right),\label{eq:d14}
\end{equation}
and 
\begin{align}
\Lambda_{3} & :=\left\{ 0,1,3,4,9,10,12,13,27,28,\cdots\right\} \label{eq:d15}\\
 & =\left\{ \sum\nolimits _{0}^{\text{finite}}b_{i}3^{i}\mathrel{;}b_{i}\in\left\{ 0,1\right\} \right\} .\nonumber 
\end{align}
In fact, Jorgensen \& Pedersen \cite{MR1655831} showed that 
\begin{equation}
\mathscr{L}\left(\nu_{3}\right)=\emptyset.\label{eq:d16}
\end{equation}
However, $\mathscr{M}\left(\Lambda_{3}\right)$ is not known. See
also \cite{MR1629847,MR1655832,MR1785282,MR2921088}.
\end{example}
\begin{figure}[H]
\begin{tabular}{>{\centering}p{0.45\paperwidth}}
\includegraphics[height=0.15\paperheight]{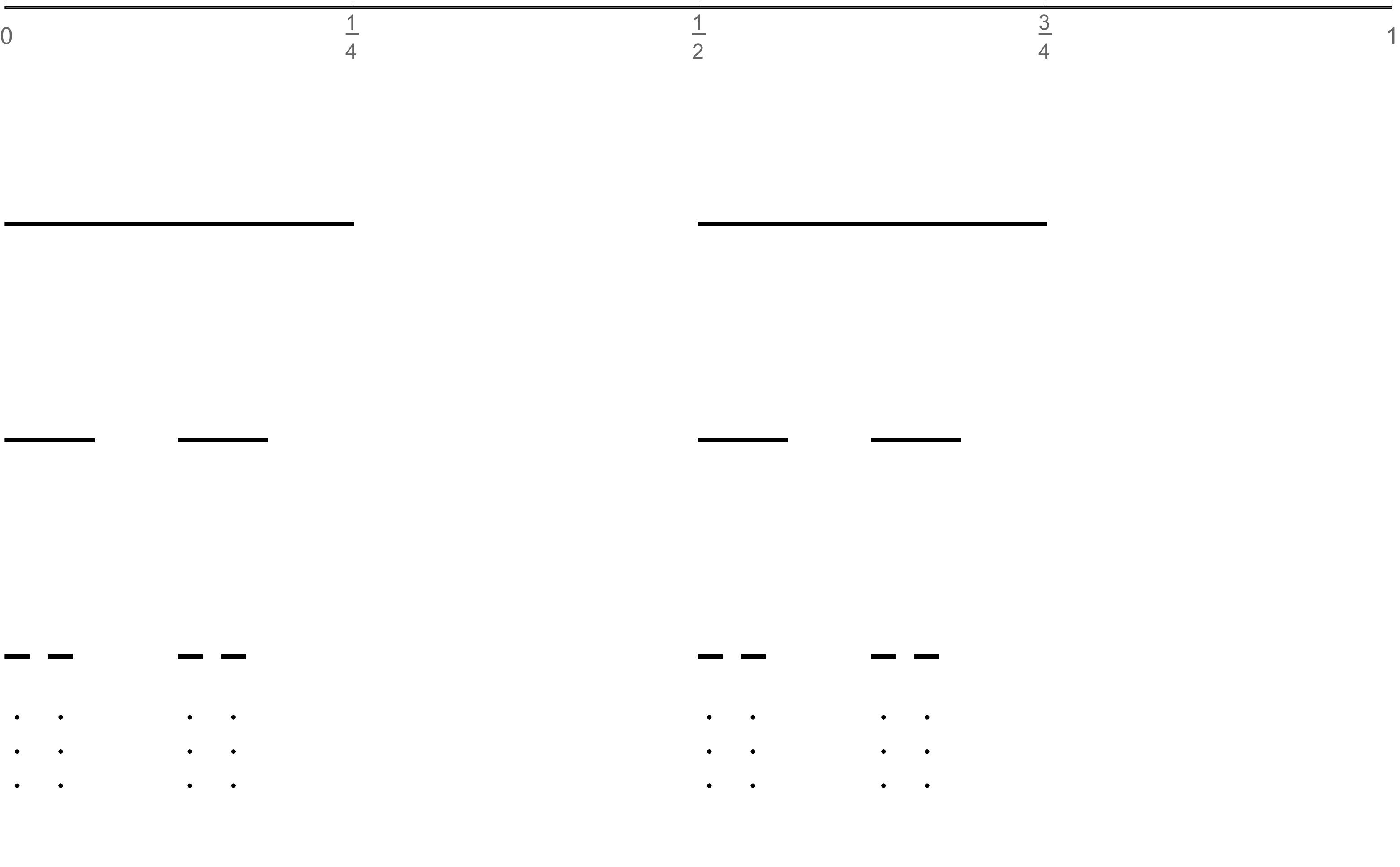}\tabularnewline
Suport of $\nu_{4}$; note $\nu_{4}$ makes a spectral pair $\left(\nu_{4},\Lambda_{4}\right)$\vspace{1em}\tabularnewline
\includegraphics[height=0.15\paperheight]{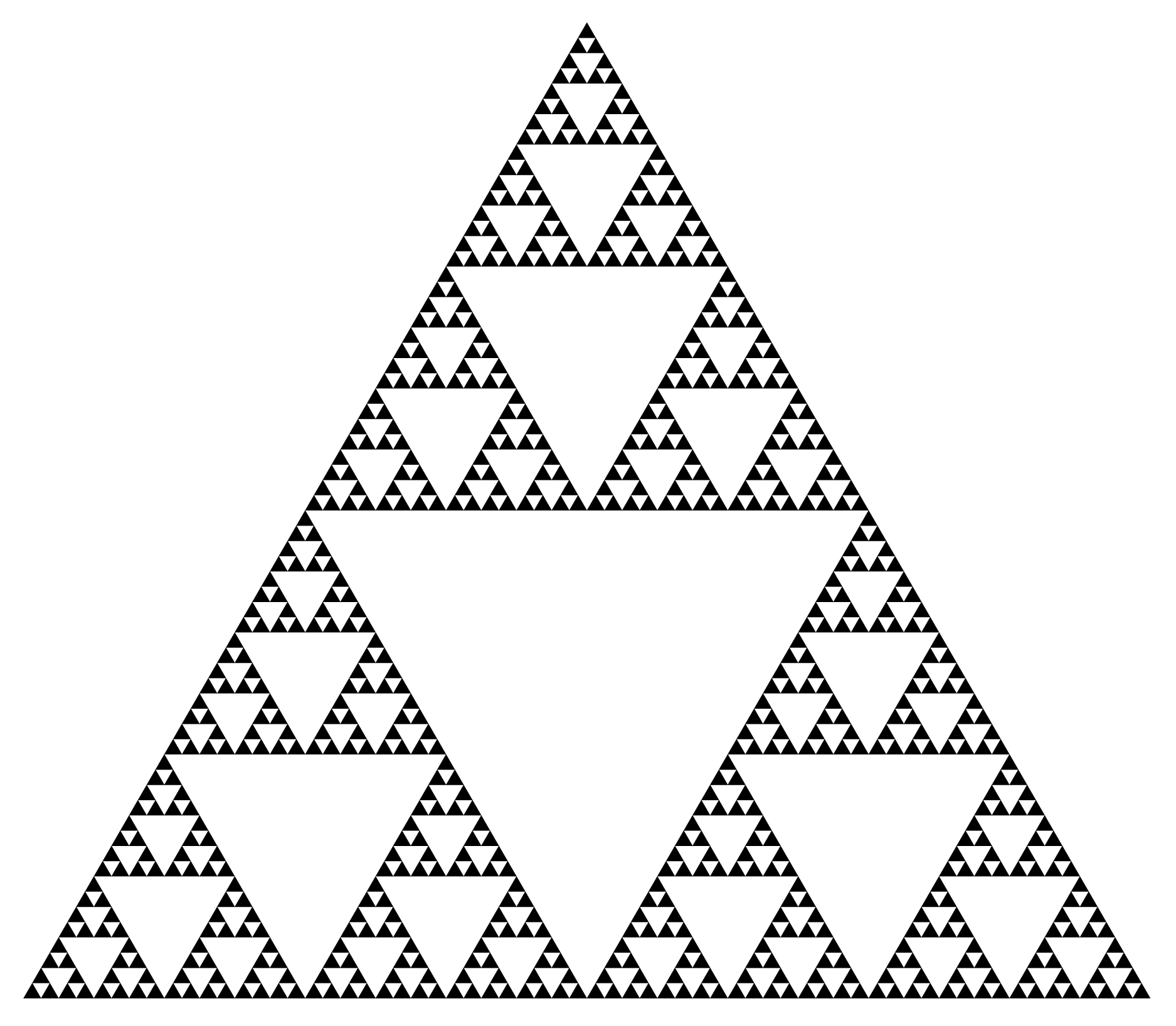}\tabularnewline
The Sierpinski gasket; the corresponding IFS measure is \emph{not}
part of a spectral pair.\vspace{1em}\tabularnewline
\includegraphics[height=0.25\paperheight]{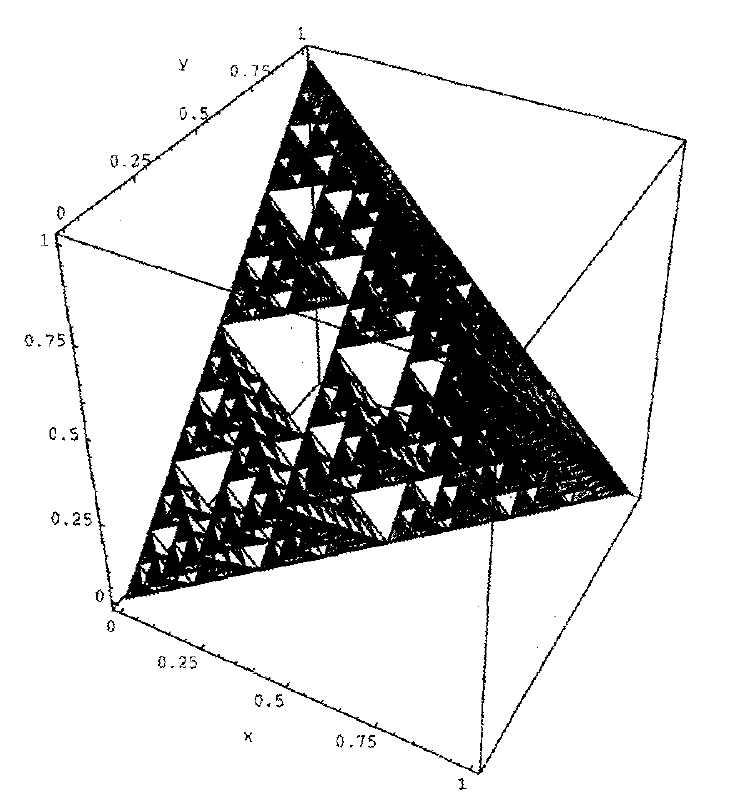}\tabularnewline
The Sierpinski Eiffle Tower; the corresponding IFS measure $\nu$
is part of a spectral pair $\left(\nu,\Lambda\right)$ with some $\Lambda\in\mathscr{D}_{3}$;
see \cite{MR1655831}, and also \cite{MR1629847,MR1655832,MR1785282,MR2921088}.\tabularnewline
\end{tabular}

\caption{\label{fig:sm}Examples of support sets in $\mathbb{R}^{k}$, $k=1,2,3$,
for IFS Cantor measures. See \cite{MR1215311}.}
\end{figure}

\begin{thm}
\label{thm:d12}Let $k\in\mathbb{N}$, and $\Lambda\in\mathscr{D}_{k}$
(see Definition \ref{def:ud}), and set $F_{\Lambda}\left(x\right)=\sum_{\lambda\in\Lambda}e^{i\lambda\cdot x}$,
$x\in\mathbb{R}^{k}$ as a tempered p.d. distribution, and let $\mathscr{H}_{F_{\Lambda}}$
be the generalized RKHS of Schwartz \cite{MR0179587}. Then a function
$h$ on $\mathbb{R}^{k}$ is in $\mathscr{H}_{F_{\Lambda}}$ if and
only if it has a convolution-factorization 
\begin{equation}
h=\varphi\ast F_{\Lambda}\label{eq:da1}
\end{equation}
where $\varphi$ is a measurable function such that $\widehat{\varphi}\left(\lambda\right)$
exists for all $\lambda\in\Lambda$, and $\left\{ \widehat{\varphi}\left(\lambda\right)\right\} _{\lambda\in\Lambda}$
is in $l^{2}\left(\Lambda\right)$. In this case 
\begin{equation}
\left\Vert h\right\Vert _{\mathscr{H}_{F_{\Lambda}}}^{2}=\sum_{\lambda\in\Lambda}\left|\widehat{\varphi}\left(\lambda\right)\right|^{2};\label{eq:da2}
\end{equation}
see (\ref{eq:da1}).
\end{thm}
\begin{proof}
We shall include below only a sketch; additional details will follow
inside the proof of Theorem \ref{thm:d3} below. The key formula for
$\left\Vert \cdot\right\Vert _{\mathscr{H}_{F_{\Lambda}}}$ is 
\begin{equation}
\left\Vert \varphi\ast F_{\Lambda}\right\Vert _{\mathscr{H}_{F_{\Lambda}}}^{2}=\sum_{\lambda\in\Lambda}\left|\widehat{\varphi}\left(\lambda\right)\right|^{2},\label{eq:da3}
\end{equation}
which in turn follows from Theorem \ref{thm:d7} above.
\end{proof}
\begin{rem}
The tempered distributions $F_{\Lambda}$ considered here play an
important role in the literature on aperiodic phenomena, especially
the study of diffraction patterns; see e.g., \cite{MR3667579,MR3485709,MR3493304,MR3600642}.
\end{rem}
\begin{thm}
\label{thm:d3}Let $\Lambda\in\mathscr{D}_{k}$, and set 
\begin{equation}
F_{\Lambda}\left(x\right):=\sum_{\lambda\in\Lambda}e^{i\lambda x},\label{eq:d17}
\end{equation}
see (\ref{eq:d11}), then $\mathscr{H}_{F_{\Lambda}}$ (eqns. (\ref{eq:d5})-(\ref{eq:d6}))
has the form $\mathscr{H}_{F_{\Lambda}}=L^{2}\left(\nu\right)$ for
a finite positive Borel measure $\nu$ on $\mathbb{R}^{k}$ if and
only if $\nu\in\mathscr{M}_{\text{Pa}}\left(\Lambda\right)$. 
\end{thm}
\begin{proof}
For $\varphi\in\mathcal{S}_{k}$, let $\widehat{\varphi}$ denote
the standard Fourier transform. It is known \cite{MR0185423} that
$\widehat{\mathcal{S}}_{k}=\mathcal{S}_{k}$, and so $\widehat{\mathcal{S}}_{k}'=\mathcal{S}_{k}'$.
One checks that 
\begin{align}
\left(\varphi\ast F_{\Lambda}\right)\left(x\right) & =\sum_{\lambda\in\Lambda}\widehat{\varphi}\left(\lambda\right)e^{i\lambda x},\;x\in\mathbb{R}^{k},\quad\text{and}\label{eq:d18}\\
\left\Vert \varphi\ast F_{\Lambda}\right\Vert _{\mathscr{H}_{F_{\Lambda}}}^{2} & =\sum_{\lambda\in\Lambda}\left|\widehat{\varphi}\left(\lambda\right)\right|^{2};\label{eq:d19}
\end{align}
see (\ref{eq:d5})-(\ref{eq:d6}) for definitions. Note that the transforms
are computed in the sense of Schwartz distributions.

We now turn to the proof of the theorem. Note that by Definition \ref{def:ud},
given $\Lambda$, then $\nu\in\mathscr{M}\left(\Lambda\right)$ iff
$L^{2}\left(\nu\right)$ admits the representation: $f\in L^{2}\left(\nu\right)\Longleftrightarrow$
$\exists\left\{ c_{\lambda}\right\} _{\lambda\in\Lambda}\in l^{2}\left(\Lambda\right)$
such that 
\begin{gather}
f\left(x\right)=\sum_{\lambda\in\Lambda}c_{\lambda}e^{i\lambda x},\quad\text{and}\label{eq:d20}\\
\int_{\mathbb{R}^{k}}\left|f\left(x\right)\right|^{2}d\nu\left(x\right)=\sum_{\lambda\in\Lambda}\left|c_{\lambda}\right|^{2},\label{eq:d21}
\end{gather}
where, for simplicity, we assume that $\nu$ is normalized. The interpretation
of (\ref{eq:d20}) is that of $L^{2}\left(\nu\right)$, and so when
$f\in L^{2}\left(\nu\right)$ has the representation (\ref{eq:d20}),
then we have convergence
\begin{equation}
\int\left|f\left(x\right)-\sum\nolimits _{\lambda\in\Lambda_{\text{fin}}}c_{\lambda}e^{i\lambda x}\right|^{2}d\nu\left(x\right)\rightarrow0\label{eq:d22}
\end{equation}
where the limit in (\ref{eq:d22}) is over the filter of all finite
subsets of the given set $\Lambda\in\mathscr{D}_{k}$. The conclusion
of the theorem now follows from this, combined with (\ref{eq:d19})-(\ref{eq:d20}),
and the previous lemmas. 

Conversely, suppose $\mathscr{H}_{F_{\Lambda}}=L^{2}\left(\nu\right)$
for some $\nu$, a finite positive Borel measure; then $e_{\lambda}\left(x\right):=e^{i\lambda x}$
is in $\mathscr{H}_{F_{\Lambda}}$ for all $\lambda\in\Lambda$. Now,
by the above, we know that every 
\begin{equation}
f\in\mathscr{H}_{F_{\Lambda}}\left(=L^{2}\left(\nu\right)\right)\label{eq:d23}
\end{equation}
has the representation 
\begin{equation}
f\left(x\right)=\sum_{\lambda\in\Lambda}\widehat{\varphi}\left(\lambda\right)e_{\lambda}\left(x\right)\label{eq:d24}
\end{equation}
where $\widehat{\varphi}\left(\lambda\right)$ denotes the standard
Fourier transform, i.e., $\widehat{\varphi}\left(\lambda\right)=\int_{\mathbb{R}^{k}}\varphi\left(x\right)e^{-i\lambda x}dx$,
and $dx=$ the $k$-dimensional Lebesgue measure; see (\ref{eq:d20}).
We now verify that $\nu\in\mathscr{M}_{\text{Pa}}\left(\Lambda\right)$,
i.e., that 
\begin{equation}
\sum\nolimits _{\lambda\in\Lambda}\left|\left\langle f,e_{\lambda}\right\rangle _{L^{2}\left(\nu\right)}\right|^{2}=\left\Vert f\right\Vert _{L^{2}\left(\nu\right)}^{2}.\label{eq:d25}
\end{equation}
We have 
\begin{alignat*}{2}
\text{RHS}_{\left(\ref{eq:d25}\right)} & =\left\Vert \sum\nolimits _{\lambda\in\Lambda}\widehat{\varphi}\left(\lambda\right)e_{\lambda}\left(\cdot\right)\right\Vert _{\mathscr{H}_{F_{\Lambda}}}^{2} & \quad & \left(\text{by }\left(\ref{eq:d23}\right)\right)\\
 & =\sum\nolimits _{\lambda\in\Lambda}\left|\widehat{\varphi}\left(\lambda\right)\right|^{2} &  & \left(\text{by }\left(\ref{eq:d21}\right)\:\text{and }\left(\ref{eq:d24}\right)\right)\\
 & =\sum\nolimits _{\lambda\in\Lambda}\left|\left\langle f,e_{\lambda}\right\rangle _{\mathscr{H}_{F_{\Lambda}}}\right|^{2}\\
 & =\sum\nolimits _{\lambda\in\Lambda}\left|\left\langle f,e_{\lambda}\right\rangle _{L^{2}\left(\nu\right)}\right|^{2} &  & \left(\text{by }\left(\ref{eq:d23}\right)\right)\\
 & =\text{LHS}_{\left(\ref{eq:d25}\right)}
\end{alignat*}
which is the desired conclusion, i.e., $\nu\in\mathscr{M}_{\text{Pa}}\left(\Lambda\right)$. 
\end{proof}
\begin{cor}
Let the setting be as in Theorem \ref{thm:d3}, i.e., $\Lambda\in\mathscr{D}_{k}$
is given, and a positive finite measure $\nu$ on $\mathbb{R}^{k}$
is assumed to satisfy $\mathscr{H}_{F_{\Lambda}}=L^{2}\left(\nu\right)$,
see Theorem \ref{thm:d7}. Then $\nu\left(\mathbb{R}^{k}\right)\leq1$,
and $\nu\left(\mathbb{R}^{k}\right)=1$ holds if and only if $\left(\nu,\Lambda\right)$
is a spectral pair, i.e., $\nu\in\mathscr{M}\left(\Lambda\right)$. 
\end{cor}
\begin{proof}
The result follows from Theorem \ref{thm:d3}, combined with the following:
\begin{lem}
Let $\mathscr{H}$ be a Hilbert space, and let $\left\{ \varphi_{j}\right\} _{j\in J}$,
$J$ countable, be a Parseval frame, i.e., 
\begin{equation}
\left\Vert h\right\Vert _{\mathscr{H}}^{2}=\sum_{j\in J}\left|\left\langle \varphi_{j},h\right\rangle _{\mathscr{H}}\right|^{2}\label{eq:da4}
\end{equation}
holds for all $h\in\mathscr{H}$. 

Then $\left\{ \varphi_{j}\right\} _{j\in J}$ is an ONB if and only
if $\left\Vert \varphi_{j}\right\Vert _{\mathscr{H}}=1$, for all
$j\in J$. In general, $\left\Vert \varphi_{j}\right\Vert _{\mathscr{H}}\leq1$,
$\forall j\in J$. 
\begin{proof}
Pick $j\in J$, and apply (\ref{eq:da4}) to $h=\varphi_{j}$, we
get 
\begin{equation}
\left\Vert \varphi_{j}\right\Vert _{\mathscr{H}}^{2}=\left\Vert \varphi_{j}\right\Vert _{\mathscr{H}}^{4}+\sum_{i\in J\backslash\left\{ j\right\} }\left|\left\langle \varphi_{i},\varphi_{j}\right\rangle _{\mathscr{H}}\right|^{2}\geq\left\Vert \varphi_{j}\right\Vert _{\mathscr{H}}^{4}.\label{eq:da5}
\end{equation}
The conclusion in the lemma is immediate from this: In particular,
$\left\Vert \varphi_{j}\right\Vert _{\mathscr{H}}=1$ holds iff the
terms on the RHS in (\ref{eq:da5}) $\left\langle \varphi_{i},\varphi_{j}\right\rangle _{\mathscr{H}}$,
$i\neq j$, $i,j\in J$, vanish.
\end{proof}
\end{lem}
The corollary follows since when $\nu$, $\Lambda$ are as specified
as stated, then 
\[
\left\Vert e_{\lambda}\right\Vert _{L^{2}\left(\nu\right)}^{2}=\int_{\mathbb{R}^{k}}\left|e_{\lambda}\right|^{2}d\nu=\nu\left(\mathbb{R}^{k}\right).
\]

\end{proof}

\section{A Result from \cite{2017arXiv170504198H}}

The conclusions in \cite{2017arXiv170504198H} introduced a new and
different approach to the study of spectral theory for $L^{2}\left(\nu\right)$
in the case where $\nu$ is a given compactly supported singular measure
on $\mathbb{R}$. Below we include a brief sketch. 

Let $\nu$ be a singular measure supported on $\left[0,1\right]=\mathbb{R}/\mathbb{Z}=$
the boundary of the disk $\mathbb{D}$, $\mathbb{D}=\left\{ z\in\mathbb{C}\mathrel{;}\left|z\right|<1\right\} $.
For $n\in\mathbb{N}_{0}=\mathbb{N}\cup\left\{ 0\right\} $, set $e_{n}\left(x\right)=e^{i2\pi nx}$,
and pass to the following Kaczmarz algorithm (see, e.g., \cite{MR2311862,MR3117886}):
\begin{align}
f_{0} & =\left\langle f,e_{0}\right\rangle _{\nu}e_{0}\nonumber \\
 & \cdots\label{eq:a1}\\
f_{n}\left(x\right) & =f_{n-1}\left(x\right)+\left\langle f-f_{n-1},e_{n}\right\rangle _{\nu}e_{n}\left(x\right).\nonumber 
\end{align}
Then the associated sequence 
\begin{align}
g_{0} & =e_{0}\nonumber \\
 & \cdots\label{eq:a2}\\
g_{n}\left(x\right) & =e_{n}\left(x\right)-\sum_{j=0}^{n-1}\underset{\widehat{\nu}\left(n-j\right)}{\underbrace{\left\langle e_{n},e_{j}\right\rangle _{\nu}}}g_{j}\left(x\right)\nonumber 
\end{align}
has the form: 
\begin{equation}
\sum_{j=0}^{n}\overline{\alpha_{n-j}}e_{j}\left(x\right)=g_{n}\left(x\right),\label{eq:a3}
\end{equation}
and 
\begin{thm}[\cite{2017arXiv170504198H}]
Let $\nu$ and $\left\{ g_{n}\right\} _{n\in\mathbb{N}_{0}}$be as
sketched; then every $f\in L^{2}\left(\nu\right)$ has the following
frame expansion: 
\begin{align}
f\left(x\right) & =\sum_{n\in\mathbb{N}_{0}}\left\langle f,g_{n}\right\rangle _{\nu}e_{n}\left(x\right)\label{eq:a4}\\
 & =\sum_{n\in\mathbb{N}_{0}}\left(\sum_{j=0}^{n}\widehat{f}\left(j\right)\alpha_{n-j}\right)e_{n}\left(x\right),\nonumber 
\end{align}
where $\left(\overline{\alpha_{n-j}}\right)$ is the $\left(\infty\times\infty\right)$-matrix
inverse to the Gramian $\left(\widehat{\nu}\left(n-j\right)\right)_{j\leq n}$. 

Moreover, the system $\left\{ g_{n}\right\} _{n\in\mathbb{N}_{0}}$
in (\ref{eq:a3}) is a Parseval frame in $L^{2}\left(\nu\right)$,
i.e., 
\[
\left\Vert f\right\Vert _{L^{2}\left(\nu\right)}^{2}=\sum_{n\in\mathbb{N}_{0}}\left|\left\langle f,g_{n}\right\rangle _{L^{2}\left(\nu\right)}\right|^{2}
\]
holds for all $f\in L^{2}\left(\nu\right)$. 
\end{thm}
\begin{proof}
Readers are refereed to \cite{2017arXiv170504198H} for proof details. 
\end{proof}
\begin{acknowledgement*}
The first named author thanks the members of the Analysis, Probability
and Mathematical Physics on Fractals at Cornell University in the
Spring of 2017 for inspiration, encouragements, and many helpful conversations.
As well as collaborations and many fruitful discussions with D. Alpay,
D. Dutkay, J. Herr, S. Pedersen, and E. Weber.
\end{acknowledgement*}
\bibliographystyle{amsalpha}
\bibliography{ref}

\end{document}